\documentclass{article}
\usepackage{enumerate,amsmath,amssymb,bm,ascmac,amsthm,url,comment}
\usepackage{longtable}
\usepackage[margin=1in]{geometry}

\theoremstyle{definition}
\newtheorem{theo}{Theorem}[section]
\newtheorem{theo*}{Theorem}

\newtheorem{defi*}{Definition}
\newtheorem{lemm}[theo]{Lemma}
\newtheorem{lemm*}{Lemma}
\newtheorem{prop}[theo]{Proposition}
\newtheorem{prop*}{Proposition}

\newtheorem{cor}[theo]{Corollary}

\newcommand{\Aut}{{\rm Aut}}
\newcommand{\Aute}{\Aut(X)_{\mathcal{E}}}
\newcommand{\zero}{{\bf 0}}
\newcommand{\wt}{{\rm wt}}
\newcommand{\even}{{\rm even}}
\newcommand{\odd}{{\rm odd}}
\newcommand{\MOD}{{\rm mod}\,\,}
\newcommand{\rank}{{\rm rank}\,}
\newcommand{\Ker}{{\rm Ker}\,}

\newcommand{\SpanS}{\langle S \rangle}
\newcommand{\cN}{\mathcal{N}}

\title{Strongly regular graphs with the same parameters \\ as the symplectic graph}
\author{Sho Kubota}
\date{}

\begin{document}
\maketitle
\begin{abstract}
We consider orbit partitions of groups of automorphisms for the symplectic graph 
and apply Godsil-McKay switching.
As a result,
we find four families of strongly regular graphs 
with the same parameters as the symplectic graphs,
including the one discovered by Abiad and Haemers.
Also,
we prove that 
switched graphs are non-isomorphic to each other 
by considering the number of common neighbors of three vertices. 
\vspace{10pt} 
\end{abstract}

\section{Introduction}

Godsil-McKay switching is often used to construct cospectral graphs.
However,
to apply that,
a partition of the vertex set of a graph has to satisfy two very strong conditions.
The orbit partition of a group of automorphisms satisfies one of them automatically,
so if we can find the orbit partition which satisfies the other one,
we can apply Godsil-McKay switching and we might be able to get cospectral graphs.

For the symplectic graph $Sp(2\nu,2)$,
Abiad and Haemers \cite{AH} considered a special 4-subset $S$ and 
the partition $\{S, V(Sp(2\nu,2)) \setminus S\}$.
And then by applying a Godsil-McKay swithcing,
they obtained many graphs with the same parameters as the symplectic graph.
We also aim to construct many graphs 
with the same parameters as the symplectic graph by applying Godsil-McKay switching,
but partitions of the vertex set we consider are 
the orbit partitions of groups of automorphisms.
In this paper,
we consider the following groups:
	\begin{itemize}
	\item The automorphism group that fixes the standard basis
	\item The automorphism group that fixes a special 4-subset by Abiad and Haemers
	\end{itemize}
As a result,
we obtain four families of strongly regular graphs 
with the same parameters as the symplectic graphs.
Also,
we see one of them is isomorphic to the one by Abiad and Haemers.
More precisely, we see the edges involved with switching are the same.

Additionally,
on the symplectic graph,
we can regard the set of common neighbors as 
the solution set of a system of linear equations.
From this point of view,
we investigate the number of common neighbors of three vertices 
as an invariant for isomorphism.
As a result,
we prove that the graphs in the five families,
which are the four switched ones and the original one,
are certainly all non-isomorphic.

\section{Preliminaries}

%

Let $\mathbb{F}_2^{2\nu}$ be the $2 \nu$-dimensional vector space over $\mathbb{F}_2$,
and let
\[ R = \begin{bmatrix} 0 & 1 \\ 1 & 0 \end{bmatrix}. \]
The {\it symplectic graph} $Sp(2\nu,2)$ over $\mathbb{F}_2$
is the graph defined by the following:
		\begin{align*}
V(Sp(2\nu,2)) &= \mathbb{F}_2^{2\nu} \setminus \{ \zero \}, \\
E(Sp(2\nu,2)) &= \{ xy \,|\, x^T K y = 1\},
		\end{align*}
where $K = I_\nu \otimes R$ ($I_\nu$ is the identity matrix of order $\nu$).
We see that $Sp(2\nu,2)$ is a strongly regular graph with parameters
$(2^{2\nu}-1,\, 2^{2\nu-1},\, 2^{2\nu-2},\, 2^{2\nu-2})$
by easy calculation.
%

In general,
the spectrum of a strongly regular graph is determined by its parameters.
Conversely,
parameters are also characterized by the spectrum.
Therefore if a graph $X'$ has the same spectrum as a strongly regular graph $X$,
then $X'$ is also strongly regular with the same parameters as $X$.
In particular,
if we can find a graph $X$
with the same spectrum as the symplectic graph $Sp(2\nu,2)$ 
which is not isomorphic to $Sp(2\nu,2)$,
then $X$ is a strongly regular graph 
with the same parameters as $Sp(2\nu,2)$
and $X$ could possibly be a new strongly regular graph.

Returning on the subject of the symplectic graph $Sp(2\nu,2)$,
when $\nu = 1$ we just have the complete graph $K_3$.
Also, $Sp(4,2)$ is a strongly regular graph with parameters $(15,8,4,4)$,
which is known to be determined by its parameters,
so we suppose that $\nu \geq 3$ in the rest of this paper.

Let $X$ be a graph 
and let $\pi = \{C_1, \dots, C_t \}$ be a partition of $V(X)$.
This partition $\pi$ is called an {\it equitable partition} if for all $i,j$
any two vertices in $C_i$ have the same number of neighbors in $C_j$.

Godsil and McKay \cite{GM} proved the following result 
on constructing cospectral graphs.

\begin{theo} {\it 
Let $X$ be a graph 
and let $\pi = \{C_1, \dots, C_t, D \}$ be a partition of $V(X)$.
Assume that $\pi$ satisfies the following two conditions: 
	\begin{enumerate}[(i)]
	\item $\{C_1, \dots, C_t \}$ is an equitable partition of $V(X) \setminus D$.
	\item For every $x \in D$ and every $i \in \{1, \dots ,t \}$
the vertex $x$ has either $0, \frac{1}{2}|C_i|$ or $|C_i|$ neighbors in $C_i$.
	\end{enumerate}
Construct a new graph $X'$ by interchanging adjacency and nonadjacency 
between $x \in D$ and the vertices in $C_i$ 
whenever $x$ has $\frac{1}{2}|C_i|$ neighbors in $C_i$.
Then $X$ and $X'$ have the same spectrum. }
\end{theo}

The operation that transforms $X$ into $X'$ is called {\it Godsil-McKay switching}.
We will call a partion $\pi$ of $V(X)$ a {\it Godsil-McKay partition}
if we can apply the above theorem with respect to $\pi$.
Also,
we will call the special cell $D$ a {\it Godsil-McKay cell} of $\pi$.

%
%

On the other hand,
the orbit partition of a subgroup of automorphisms of a graph forms 
an equitable partition,
so this automatically satisfies the condition (i) of 
Godsil-McKay switching no matter what orbit we choose as $D$.

Tang and Wan \cite{TW} determined the automorphism group of $Sp(2\nu,2)$.

\begin{prop} {\it
\[ \Aut(Sp(2\nu,2)) \simeq Sp_{2\nu}(\mathbb{F}_2) \]
where $Sp_{2\nu}(\mathbb{F}_2) = \{ A \in GL_{2\nu}(\mathbb{F}_2) \, | \, A^T K A = K \}.$}
\end{prop}

However, $Sp(2\nu,2)$ is vertex-transitive.
We aim to find Godsil-McKay cells in the orbit partition of a group of automorphisms,
so we have to choose a proper subgroup of the automorphism group.


\section{Automorphisms that fix the standard basis}

In this section,
we consider the subgroup of automorphisms that
fix the set of the standard basis of $\mathbb{F}_2^{2\nu}$.
To apply Godsil-McKay switching,
we determine the orbit partition and
confirm that it is a Godsil-McKay partition.
After that,
we prove that a switched graph is not isomorphic to the original symplectic graph.

Let $X$ be the symplectic graph $Sp(2\nu,2)$ of order $2\nu$
and let $e_i$ be the vector in $\mathbb{F}_2^{2\nu}$
with a 1 in the $i$th coordinate and 0's elsewhere and
put $\mathcal{E} = \{ e_1, \dots, e_{2\nu} \}$.
Also, let 
\[ \Aute = \{ g \in \Aut(X) \, | \, \mathcal{E}^g = \mathcal{E} \} . \]

\subsection{Determination of the orbit partition of $\Aute$}

Let $P$ be a permutation matrix of order $\nu$ and
let $A_1, \dots, A_{\nu}$ be matrices of order 2.
We define the matrix $P(A_1,\dots,A_{\nu})$ of order $2\nu$ as follows:
	\begin{align*}
P_{ij} \mapsto
\begin{cases} 
O \quad &\text{if $P_{ij} = 0$,} \\
A_i \quad &\text{if $P_{ij} = 1$,}
\end{cases}
	\end{align*}
where $O$ is the zero matrix.
Note that $K=I_{\nu}(R,\dots,R)$.

\begin{lemm}{\it
Let $A_1,\dots,A_{\nu}$ be matrices of order 2 over $\mathbb{F}_2$
and set
\[ B = \begin{bmatrix}
A_1 \\
\vdots \\
A_{\nu}
\end{bmatrix}.
 \]
Suppose the two column vectors $b_1,b_2$ of $B$ satisfy the following conditions:
	\begin{itemize}
	\item $b_1 \neq b_2$,
	\item $\wt(b_1) = \wt(b_2) = 1$,
where the weight $\wt(x)$ of a vector $x$ is the number of non-zero components of $x$.
	\end{itemize}
If $A_1^T R A_1 + \dots + A_{\nu}^T R A_{\nu} = R$,
then there exists a unique $i \in [\nu]$ such that 
$A_i \in \{ I_2,\, R \}$ and $A_j = O$ for all $j \in [\nu] \setminus \{i\}$.
}
\end{lemm}

\begin{proof}
By the second condition on $B$,
there exists $i \in [\nu]$ such that $A_i \neq O$ and 
the number of components of 1 in $A_i$ is 1 or 2.

Case 1: Suppose that the number of components of 1 in $A_i$ is 1.
There exists another $j \in [\nu] \setminus \{ i \}$ such that $A_j \neq O$.
By the second condition on $B$,
the number of components of 1 in $A_j$ has to be 1
and $A_k = O$ for all $k \in [\nu] \setminus \{ i,j \}$.
However $A_i^T R A_i = A_j^T R A_j = O$
and clearly $A_k^T R A_k = O$ for all $k \in [\nu] \setminus \{ i,j \}$.
Therefore $A_1^T R A_1 + \dots + A_{\nu}^T R A_{\nu} = O$.
This is a contradiction.

Case 2: Suppose that the number of components of 1 in $A_i$ is 2.
By the two conditions on $B$,
we have $A_i = I_2$ or $R$.
Moreover by the second condition on $B$,
it follows that $A_j = O$ for all $j \in [\nu] \setminus \{i\}$.
\end{proof}

\begin{lemm} {\it
\[ \Aute \simeq
\left\{ P(A_1,\dots, A_{\nu}) \, \Big| \, 
\text{P : permutation matrix}, A_i \in \{I_2,R\} \right\}. \]
}
\end{lemm}

\begin{proof}
Put $\mathcal{P} =
\left\{ P(A_1, \dots, A_{\nu}) \, \Big| \, 
\text{$P$ : permutation matrix}, A_i \in \{I_2,R\} \right\}$.
By Proposition~2.2, 
\[ \Aute \simeq
\{ A \in GL_{2\nu}(\mathbb{F}_2) \, | \, A^T K A = K, A \mathcal{E} = \mathcal{E} \}, \]
so we set $Sp_{2\nu}(\mathbb{F}_2)_{\mathcal{E}}
= \{ A \in GL_{2\nu}(\mathbb{F}_2) \, | \, A^T K A = K, A \mathcal{E} = \mathcal{E} \}$
and prove that $\mathcal{P} = Sp_{2\nu}(\mathbb{F}_2)_{\mathcal{E}}$.

First, let $P(A_1,\dots,A_{\nu}) \in \mathcal{P}$.
Since $P(A_1,\dots,A_{\nu})$ is a permutation matrix of order $2\nu$,
$P(A_1,\dots,A_{\nu}) \mathcal{E} = \mathcal{E}$.
Also,
	\begin{align*}
P(A_1,\dots,A_{\nu})^T K P(A_1,\dots,A_{\nu}) 
&= P^T(A_1^T,\dots,A_{\nu}^T) K P(A_1,\dots,A_{\nu}) \\
&= P^T(A_1,\dots,A_{\nu}) K P(A_1,\dots,A_{\nu}) \\
&= P^T(A_1,\dots,A_{\nu}) I_{\nu}(R,\dots,R) P(A_1,\dots,A_{\nu}) \\
&= (P^T I_{\nu} P) (A_1 R A_1,\dots,A_{\nu} R A_{\nu}) \\
&= I_{\nu} (A_1 R A_1,\dots,A_{\nu} R A_{\nu}) \\
& = I_{\nu} (R, \dots, R) \\
& = K.
	\end{align*}
Consequently,
we see that $P(A_1,\dots,A_{\nu}) \in Sp_{2\nu}(\mathbb{F}_2)_{\mathcal{E}}$.

Conversely, let $A \in Sp_{2\nu}(\mathbb{F}_2)_{\mathcal{E}}$.
Since $A \mathcal{E} = \mathcal{E}$,
$A$ is a permutation matrix.
We set $A$ as a block matrix as follows:
\[ A = \begin{bmatrix}
A_{11}& \cdots & A_{1 \nu} \\
\vdots&\ddots&\vdots \\
A_{\nu 1}&\cdots& A_{\nu \nu}
\end{bmatrix} \qquad (A_{ij} \in M_2(\mathbb{F}_2)). \]
By $A^T K A = K$, we get
\[
\begin{bmatrix}
\sum_{i = 1}^{\nu} A_{i1}^T R A_{i1}& & & \\
&\sum_{i = 1}^{\nu} A_{i2}^T R A_{i2}&& \\
 && \ddots & \\  
 & & & \sum_{i = 1}^{\nu} A_{i\nu}^T R A_{i\nu} 
\end{bmatrix}
=
\begin{bmatrix}
R&&&\\
&R&&\\
&&\ddots&\\
&&&R
\end{bmatrix}.
\]
By comparing the (1,1) blocks,
we have $\sum_{i = 1}^{\nu} A_{i1}^T R A_{i1} = R$.
Since $A$ is a permutation matrix,
weights of the two column vectors of
\[
\begin{bmatrix}
A_{11} \\
\vdots \\
A_{\nu 1}
\end{bmatrix}
\]
are both 1 and these two column vectors are distinct.
Therefore we can apply Lemma~3.1,
that is,
there exists a unique $i_1 \in [\nu]$ such that
$A_{i_1,1} = I_2$ or $R$ and $A_{j1} = O \, (j \neq i_1)$.
Moreover, $A$ is a permutation matrix,
so we get $A_{i_1, 2} = A_{i_1, 3} = \dots = A_{i_1, \nu} = O$.
Next, we compare the (2,2) blocks.
By the similar argument above,
we see that there exists a unique $i_2 \in [\nu] \setminus \{ i_1 \}$ such that
$A_{i_2,2} = I_2$ or $R$ and
$A_{j2} = O \, (j \neq i_2)$,
so we get $A_{i_2, 3} = A_{i_2, 4} = \dots = A_{i_2, \nu} = O$.
Continuing this argument repeatedly,
we eventually have a permutation matrix $P$ of order $\nu$ and
$\nu$ matrices $B_i \in \{I_2, R\} \, (i = 1,\dots,\nu)$ such that
$A = P(B_1,\dots,B_{\nu})$.
Consequently, we see that $A \in \mathcal{P}$.
\end{proof}

Hereafter,
we often divide a vector $x \in \mathbb{F}_2^{2\nu}$
into $\nu$ blocks as follows:
\[ x = \begin{bmatrix}
x_1 \\ \vdots \\ x_{\nu}
\end{bmatrix} \qquad (x_i \in \mathbb{F}_2^2). \]
We define 
\[ O(i,j,k) = \left\{ 
x = \begin{bmatrix}
x_1 \\
\vdots \\
x_{\nu} \end{bmatrix} \in \mathbb{F}_2^{2\nu} \, \Bigg | \,
\begin{split}
\#\{ l \,|\, \wt(x_l) = 2\} = i, \\
\#\{ l \,|\, \wt(x_l) = 1\} = j, \\
\#\{ l \,|\, \wt(x_l) = 0\} = k
\end{split}
 \right\}. \]
Note that $i+j+k = \nu$.
A vector $x$ is called the {\it initial vector} of $O(i,j,k)$ if
$x_1 = \dots = x_i = [11]^T,
x_{i+1} = \dots = x_{i+j} = [10]^T,
x_{i+j+1} = \dots = x_{i+j+k} = [00]^T$.
For example,
the initial vector of $O(1,2,1)$ is $[11101000]^T$.

\begin{prop} {\it
Let $X$ be the symplectic graph $Sp(2\nu,2)$ of order $2\nu$.
The orbit partition of $\Aute$ on $V(X)$ is the following:
\[ V(X) =
\bigsqcup_{\substack{i+j+k = \nu, \\ (i,j,k) \neq (0,0,\nu) }} O(i,j,k). \]
}
\end{prop}

\begin{proof}
First,
we prove that for all $x,y \in O(i,j,k)$ there exists $g \in \Aute$ such that $y = x^g$,
but we can assume that $y$ is the initial vector of $O(i,j,k)$ without loss of generality.
Let $S = \{ i \in [\nu] \, | \, x_i = [01]^T \}$
and consider the matrix $A$ of order $2\nu$ defined by the following:
\[ A = I_{\nu}(A_1,\dots,A_{\nu}), \quad \text{where }
A_i = \begin{cases} R \quad \text{if $i \in S$}, \\ I_2 \quad \text{otherwise}. \end{cases} \]
Then all weight-one blocks of $Ax$ are $[10]^T$.
After that,
we can choose an appropriate permutation matrix $P$ of order $\nu$ such that
the wights of the $\nu$ blocks of the vector $P(I_2,\dots,I_2)Ax$ are in decreasing order.
Then $P(I_2,\dots,I_2)Ax$ is nothing but the initial vector of $O(i,j,k)$,
that is, $y = P(I_2,\dots,I_2)Ax$.
By Lemma~3.2 and Proposition~2.2,
the mapping 
\[ T_{P(I_2,\dots,I_2)A} : x \mapsto P(I_2,\dots,I_2)Ax \]
is certainly an automorphism that fixes the standard basis.

Next,
we prove that for each $x \in O(i,j,k)$ and $y \in O(i',j',k')$ with $(i,j,k) \neq (i',j',k')$,
$y \neq x^g$ for all $g \in \Aute$.
By Lemma~3.2 and Proposition~2.2,
for $g \in \Aute$ there exists $P(A_1,\dots,A_{\nu}) \in \mathcal{P}$ such that
$g = T_{P(A_1,\dots,A_{\nu})}$,
where $T_{P(A_1,\dots,A_{\nu})}$ is the mapping that maps 
$z \in V(X)$ to $P(A_1,\dots,A_{\nu})z$. 
However,
roles which $P(A_1,\dots,A_{\nu})$ plays are only permuting blocks 
and exchanging the components of a block,
so if $x \in O(i,j,k)$ then $P(A_1,\dots,A_{\nu})x \in O(i,j,k)$.
Therefore, it follows that $y \neq x^g$ for all $g \in \Aute$.
\end{proof}

\subsection{Finding Godsil-McKay cells in orbit partitions}

We define $O(i,j,k)_{\even}$ and $O(i,j,k)_{\odd}$ as follows,
to decompose $O(i,j,k)$ into two more sets:
	\begin{align*}
O(i,j,k)_{\even} = \left\{
x = \begin{bmatrix} x_1 \\ \vdots \\ x_{\nu} \end{bmatrix} \in O(i,j,k)
\, \Bigg | \,
\# \left\{ l \in [\nu] \, \Big | \,
x_l = \begin{bmatrix} 1 \\ 0 \end{bmatrix} \right\} = \even \right\}, \\
O(i,j,k)_{\odd} = \left\{
x = \begin{bmatrix} x_1 \\ \vdots \\ x_{\nu} \end{bmatrix} \in O(i,j,k)
\, \Bigg | \,
\# \left\{ l \in [\nu] \, \Big | \,
x_l = \begin{bmatrix} 1 \\ 0 \end{bmatrix} \right\} = \odd \right\}.
	\end{align*}
Actually,
we can see that there exists a bijection between $O(i,j,k)_{\even}$ and $O(i,j,k)_{\odd}$.

\begin{lemm}
$|O(i,j,k)_{\even}|=|O(i,j,k)_{\odd}|$.
\end{lemm}

\begin{proof}
If $j =0$,
$O(i,j,k)_{\even}$ and $O(i,j,k)_{\odd}$ are empty sets,
so we have the above equality.
Suppose that $j \geq 1$.
For $x \in O(i,j,k)$, 
we can define $l_{\min} = \min \{ l \in [\nu] \,|\, \wt(x_l) = 1 \}$. 
Consider the following correspondence:
\[ x = \begin{bmatrix} x_1 \\ \vdots \\ x_{l_{\min}} \\ \vdots \\ x_{\nu} \end{bmatrix}
\mapsto
\begin{bmatrix} x_1 \\ \vdots \\ x_{l_{\min}} + {\bf 1}_2 \\ \vdots \\ x_{\nu}
\end{bmatrix}, \]
where ${\bf 1}_2 = [11]^T$.
By this correspondence,
parity of the number of blocks of $[10]^T$ change.
Consequently,
we get two mappings which are the one from $O(i,j,k)_{\even}$ to $O(i,j,k)_{\odd}$
and the other from $O(i,j,k)_{\odd}$ to $O(i,j,k)_{\even}$.
Clearly, these are the inverse mappings each other,
so we have the desired equality.
\end{proof}

Let $N(x)$ denote the set of all neighbors of a vertex $x$.

\begin{prop}{\it
For all $x \in O(0,\nu,0)$ and an arbitrary orbit $O(i,j,k)$
\[ |N(x) \cap O(i,j,k)| = \begin{cases}
\frac{1}{2}|O(i,j,k)| &\text{ if $j \geq 1$,} \\
|O(i,j,k)| &\text{ if $j = 0$, $i : \odd$,} \\
0 &\text{ if $j = 0$, $i : \even$.}
\end{cases}
\]
In particular,
$O(0,\nu,0)$ is a Godsil-McKay cell in the orbit partition of $\Aute$. }
\end{prop}

\begin{proof}
For $x \in O(0,\nu,0)$ and $g \in \Aute$,
since $O(i,j,k)$ is an orbit,
\[ | N(x) \cap O(i,j,k) | = | N(x^g) \cap O(i,j,k) |, \]
so we can assume that $x = [0101 \dots 01]^T$ 
as a special vertex in $O(0,\nu,0)$ without loss of generality.
Then for $y \in N(x) \cap O(i,j,k)$,
\begin{align*}
1 &= x^T K y \\
&= [1010 \dots 10] \begin{bmatrix} y_1 \\ \vdots \\ y_{\nu} \end{bmatrix} \\
&= [10]y_1 + \dots + [10]y_{\nu}, \\
\intertext{and $[10]y_l = 1$ if and only if $y_l = [11]^T$ or $[10]^T$, so we get}
1 &\equiv \# \{l \in [\nu] \,|\, y_l = [11]^T \} + \# \{l \in [\nu] \,|\, y_l = [10]^T \}
\quad (\MOD 2) \\
&= i + \# \{l \in [\nu] \,|\, y_l = [10]^T \}.
\end{align*}
We consider two cases:

Case 1 : Suppose $j = 0$.
Then $\# \{l \in [\nu] \,|\, y_l = [10]^T \} = 0$,
so $x^T K y = 1$ if and only if $i \equiv 1 \, (\MOD 2)$ by the above observation.
Therefore,
\begin{align*}
| N(x) \cap O(i,j,k) | &= \# \{ y \in O(i,j,k) \,|\, x^T K y = 1 \} \\
&= \begin{cases}
|O(i,j,k)| &\text{if $i : \odd$}, \\
0  &\text{otherwise}.
\end{cases}
\end{align*}

Case 2 : Suppose $j \geq 1$.
Similarly,
\begin{align*}
&| N(x) \cap O(i,j,k) | \\
=& \# \{ y \in O(i,j,k) \,|\, x^T K y = 1 \} \\
=& \# \left\{ y \in O(i,j,k) \,\Big|\, i + \# \{l \in [\nu] \,|\, y_l = [10]^T \} \equiv 1 \, (\MOD 2) \right\} \\
=& \begin{cases}
\# \left\{ y \in O(i,j,k) \,\Big|\, \# \{l \in [\nu] \,|\, y_l = [10]^T \} \equiv 0 \, (\MOD 2) \right\} 
&\text{if $i : \odd$}, \\[5pt]
\# \left\{ y \in O(i,j,k) \,\Big|\, \# \{l \in [\nu] \,|\, y_l = [10]^T \} \equiv 1 \, (\MOD 2) \right\} 
&\text{otherwise}, \\
\end{cases} \\
=& \begin{cases}
|O(i,j,k)_{\even}| &\text{if $i : \odd$}, \\
|O(i,j,k)_{\odd}| &\text{otherwise}, \\
\end{cases} \\
=& \frac{1}{2}|O(i,j,k)|
\end{align*}
by Lemma~3.4.
\end{proof}

By Proposition~3.5,
we can apply Godsil-McKay switching to the symplectic graph 
with respect to the orbit partition of $\Aute$ 
with the Godsil-McKay cell $O(0,\nu,0)$.
We denote this switched graph by $X^{O(0,\nu,0)}$.
We will see that $X^{O(0,\nu,0)}$ is not isomorphic 
to the original graph $Sp(2\nu,2)$ in Section~5.

\section{Automorphisms that fix a 4-subset}

Let $X$ be a graph and let $\{C_1, V(X) \setminus C_1\}$ be a partition of $V(X)$.
If $|C_1| = 2$,
then the partition $\{C_1, V(X) \setminus C_1\}$ is always a Godsil-McKay partition 
with Godsil-McKay cell $V(X) \setminus C_1$,
but the graph switched by this partition is always isomorphic to the original one.
On the other hand, if $|C_1| \geq 4$,
the partition $\{C_1, V(X) \setminus C_1\}$ is not always a Godsil-McKay partition,
but if it is, 
then the switched graph can be non-isomorphic to the original one.
In regard to this,
Abiad, Brouwer and Haemers \cite{ABH} studied 
some sufficient conditions for being non-isomorphic after Godsil-McKay switching,
but nobody knows on necessary and sufficient conditions so far.

Even so,
Abiad and Haemers \cite{AH} studied switched symplectic graphs 
with respect to partitions of the form 
$\{C_1, V(Sp(2\nu,2)) \setminus C_1\}$ with $|C_1| = 4$
and they obtained many graphs with the same parameters as $Sp(2\nu,2)$.

In this section,
we consider the subgroup of automorphisms that fix their 4-subset $C_1$.
As a result,
we find three Godsil-McKay cells including $C_1$.

Let $S = \{v_1,v_2,v_3,v_4\}$ be a 4-subset of $V(Sp(2\nu,2))$ 
satisfying the following two conditions:
	\begin{itemize}
	\item $v_1,v_2,v_3$ are linearly independent with $v_i^T K v_j = 0$ for all $i,j \in [3]$,
	\item $v_4 = v_1+v_2+v_3$.
	\end{itemize}
Note that any three vectors $v_i,v_j,v_k \in S$ are linearly independent and 
for any $x \in V(Sp(2\nu,2))$,
$x^T K v_1 + x^T K v_2 + x^T K v_3 + x^T K v_4 = x^T K \zero = 0$,
so
$\#\{ i \in [4] \,|\, x^T K v_i = 1 \} = 0,2,4$.
Therefore we can decompose $V(Sp(2\nu,2))$ into three subsets as follows:
\[ V(Sp(2\nu,2)) = S_0 \sqcup S_2 \sqcup S_4, \]
where $S_i = \left\{ x \in V(Sp(2\nu,2)) \, \Big| \, \#\{ j \in [4] \,|\, x^T K v_j = 1 \} =i \right\}$.

\subsection{Determination of the orbit partition of $\Aut(X)_S$}

Let $X$ be the symplectic graph $Sp(2\nu,2)$ and
$S$ be the above 4-subset.
We consider 
\[ \Aut(X)_S = \{ g \in \Aut(X) \,|\, S^g = S \}. \]
Let $\SpanS$ denote the subspace spanned by $S$.
By Proposition~2.2,
we get the following:

\begin{lemm}{\it
$\SpanS^g = \SpanS$ for all $g \in \Aut(X)_S$. }
\end{lemm}

Before determining the orbit partition,
we recall the useful theorem known as Witt's theorem (see for example \cite{A}).

\begin{theo}{\it
Let $V$ and $V'$ be vector spaces equipped with a non-degenerate 
symplectic inner product and 
suppose that they are isometric.
Let $\sigma$ be an isometry from an arbitrary subspace $U$ of $V$ to $V'$.
Then $\sigma$ can be extended to a surjective isometry from $V$ to $V'$.}
\end{theo}

We can regard the value of $x^T K y$ as the value of an inner product $(x,y)$,
and preserving the value of the inner product is nothing but 
preserving the adjacency relation.
Therefore Witt's theorem guarantees that
an isometry constructed from a small subspace of $\mathbb{F}_2^{2\nu}$ 
can be extended to an automorphism of $Sp(2\nu,2)$.
This is a really strong tool to prove that 
any two vertices in a set, where we want to show it is an orbit,
can be mapped to each other by an automorphism.

Let $T = \SpanS \setminus (S \cup \{\zero\}) = \{v_1+v_2,v_2+v_3,v_3+v_1\}$. 
Note that $S,T \subset S_0$.

\begin{lemm}{\it $\Aut(X)_S$ acts on $S$ as ${\rm Sym}(S)$,
where ${\rm Sym}(S)$ is the symmetric group on $S$.}
\end{lemm}

\begin{proof} 
Let $[4] = \{i_1,i_2,i_3,i_4\}$.
We consider the subspace $U = \langle v_1,v_2,v_3 \rangle$ and
the linear mapping $g$ from $U$ to $\mathbb{F}_2^{2\nu}$ such that 
$v_{j}^g = v_{i_j}$ for $j \in [3]$.
We see that $g$ is an isometry,
so there exists an automorphism $g^*$ of $Sp(2\nu,2)$ 
such that $g^* |_{U} = g$ by Witt's theorem.
Also, 
$v_{4}^g = v_{i_4}$ since $v_{i_4} = v_{i_1} + v_{i_2} + v_{i_3}$,
so we see that $\Aut(X)_S$ acts on $S$ as ${\rm Sym}(S)$. 
\end{proof}

\begin{prop} {\it
Let $X$ be the symplectic graph $Sp(2\nu,2)$.
The orbit partition of $\Aut(X)_S$ on $V(X)$ is 
$\{ S, T, S_0 \setminus (S \cup T), S_2, S_4 \}$. }
\end{prop}

\begin{proof}
First,
we prove that any two vertices in different sets cannot be mapped to each other.
For any $g \in \Aut(X)_S$ and for any $x \in V(X)$,
we can define a mapping from 
$\{ v_i \in S \,|\, x^T K v_i = 1 \}$ to $\{ v_i \in S \,|\, (x^g)^T K v_i = 1 \}$ 
that maps $v_i$ to $v_i^g$ and it is clearly bijective,
so the value of $\#\{ v_i \in S \,|\, x^T K v_i = 1 \}$ is invariant under $g \in \Aut(X)_S$.
Therefore $S_0,S_2,S_4$ cannot be mapped to each other.
By Lemma~4.1,
$\SpanS \setminus \{ \zero \}$ and $S_0 \setminus \SpanS$ 
cannot be mapped to each other.
Since $S^g = S$ for any $g \in \Aut(X)_S$,
$S$ and $T$ cannot be mapped to each other.
Consequently,
we see that $S,T,S_0 \setminus (S \cup T),S_2,S_4 $ cannot be mapped to each other.

Next,
we prove that for every $P \in \{ S,T,S_0 \setminus (S \cup T),S_2,S_4 \}$,
any two vertices in $P$ can be mapped to each other by some $g \in \Aut(X)_S$.
It is clear in the case $P \in \{S,T\}$ by Lemma~4.3.
Thus, we consider $P \in \{ S_0 \setminus (S \cup T), S_2, S_4 \}$.
Note that for three distinct vertices $v_i,v_j,v_k \in S$ 
and $x \in V(X) \setminus \SpanS$,
$x,v_i,v_j,v_k$ are linearly independent.
Assume that $P \in \{S_0 \setminus (S \cup T), S_4\}$.
Let $x,y \in P$ and we consider the subspace $U = \langle x,v_1,v_2,v_3 \rangle$ and 
the linear mapping $g$ from $U$ to $\mathbb{F}_2^{2\nu}$ such that $x^g = y$ 
and $v_i^g = v_i$ for $i \in [3]$.
Then $g$ preserves the value of the inner product 
and $g$ is injective since $x,v_1,v_2,v_3$ are linearly independent,
so $g$ is an isometry.
Therefore by Witt's theorem,
there exists an automorphism $g^*$ of $X$ such that $g^* |_U = g$.
This fixes $S$ and maps $x$ to $y$.
The case $P = S_2$ is proved by similar argument.
\end{proof}

\subsection{Finding Godsil-McKay cells}

Let $x \in V(Sp(2\nu,2))$.
Since $v_4 = v_1+v_2+v_3$,
$\#\{i \in [4] \, | \, v_i ^T K x = 1\} =4$ if and only if 
$v_1 ^T K x = v_2 ^T K x = v_3 ^T K x = 1$.
Let 
\[ M = \begin{bmatrix} v_1^T K \\ v_2^T K \\ v_3^T K \end{bmatrix}. \]
Then $S_4 = \{ x \in V(Sp(2\nu,2)) \, | \, Mx = {\bm 1}_3 \}$.
Since $v_1,v_2,v_3$ are linearly independent,
$\rank M = 3$,
so the system of equations $Mx = {\bm 1}_3$ has a solution.
Thus,
we have a bijection from $S_4$ to $\Ker T_M$,
so $|S_4| = 2^{2\nu-3}$.
A similar argument gives us $|S_0| = 2^{2\nu-3}-1$.
Also,
$S_2$ is the complement of $S_0 \cup S_4$,
so we obtain $|S_2| = (2^{2\nu}-1) - (2^{2\nu-3}-1) - 2^{2\nu-3} = 3 \cdot 2^{2\nu-2}$.
Summarizing above,
we get the following:

\begin{lemm}{\it
\[ |S_i| = \begin{cases} 
2^{2\nu-3}-1 &\text{if $i=0$,} \\
3 \cdot 2^{2\nu-2} &\text{if $i=2$,} \\
2^{2\nu-3} &\text{if $i=4$.}
\end{cases} \]}
\end{lemm}

We decompose $S_2$ more.
For distinct indices $i,j$,
define 
\[ S_2(i,j) = \{ x \in S_2 \,|\, x^T K v_i = x^T K v_j = 1 \}. \]
By Lemma~4.3,
we see that there is a bijection from $S_2(1,2)$ to $S_2(i,j)$,
so
	\begin{equation}\label{4S2}
|S_2| = \sum_{i,j} |S_2(i,j)| = 6|S_2(1,2)|.
	\end{equation}

Let $X$ be a graph and 
let $\{O_1,\dots,O_t\}$ be an orbit partition of a group of automorphisms of $X$.
Then for all $x \in O_i$,
$|N(x) \cap O_j|$ is a constant value.
By counting the cardinality of 
$\{ xy \in E(X) \,|\, x \in O_i, y \in O_j \}$ in two ways,
we obtain the following useful formula:

\begin{lemm}{\it 
For any $x \in O_i$ and $y \in O_j$,
\[ |O_i| |N(x) \cap O_j| = |O_j| |N(y) \cap O_i|.\]
}
\end{lemm}

\begin{prop}{\it 
Let $P \in \{ S,T,S_0 \setminus (S \cup T), S_2,S_4 \}$.
Then the following statements hold:
	\begin{enumerate}[(i)]
	\item {\it For any $x \in S$,
\[ |N(x) \cap P| =
\begin{cases}
0 & \text{ if $P = T$ or $S_0 \setminus (S \cup T)$, } \\
|P| & \text{ if $P = S_4$, } \\
\frac12 |P| & \text{ if $P = S_2$. }
\end{cases} \] }
	\item {\it For any $x \in S_4$,
\[ |N(x) \cap P| =
\begin{cases}
0 & \text{ if $P = T$, } \\
|P| & \text{ if $P = S$, } \\
\frac12 |P| & \text{ if $P = S_2$ or $S_0 \setminus (S \cup T)$. }
\end{cases} \] }
	\item {\it For any $x \in S_0 \setminus (S \cup T)$,
\[ |N(x) \cap P| =
\begin{cases}
0 & \text{ if $P = S$ or $T$, } \\
\frac12 |P| & \text{ if $P = S_2$ or $S_4$. }
\end{cases} \] }
	\end{enumerate}
In particular,
$S, S_0 \setminus (S \cup T), S_4$ are Godsil-McKay cells.  }
\end{prop}

\begin{proof}
First,
we prove that $S$ is a Godsil-McKay cell,
but since $\{S,T,S_0 \setminus (S \cup T),S_2,S_4\}$ is the orbit partition,
it is sufficient to prove only that for all $P \in \{T,S_0 \setminus (S \cup T),S_2,S_4\}$ 
and a special vertex $x_0 \in S$,
$|N(x_0) \cap P| = 0,\frac{1}{2}|P|$ or $|P|$.
We take $v_1 \in S$ as a special vertex.
It is easy to see that $N(v_1) \cap S_4 = S_4$ and $N(v_1) \cap S_0 = \emptyset$,
so we get $|N(v_1) \cap S_4| = |S_4|$ and 
$|N(v_1) \cap T| = |N(v_1) \cap (S_0 \setminus (S \cup T)) | = 0$.
If $i < j$, then
\[ N(v_1) \cap S_2(i,j) =
\begin{cases}
S_2(i,j) & \quad \text{if $i = 1$,} \\
\emptyset & \quad \text{otherwise.}
\end{cases} \]
Consequently, we have 
\[ N(v_1) \cap S_2 = \bigsqcup_{i,j}(N(v_1) \cap S_2(i,j)) 
= \bigsqcup_{j=2}^{4} S_2(1,j), \]
so $|N(v_1) \cap S_2| = 3 |S_2(1,2)| = \frac{1}{2}|S_2|$ by the equality (\ref{4S2}).

Next,
we prove that $S_4$ is a Godsil-McKay cell.
Let $x \in S_4$ be a special vertex.
It is easy to see that $|N(x) \cap S| = |S|$ and $|N(x) \cap T| = 0$.
To find the value of $|N(x) \cap (S_0 \setminus (S \cup T))|$,
we calculate $|N(x) \cap S_0|$ first.
Observe $N(x) \cap S_0 = 
\left\{ y \in V(Sp(2\nu,2)) \, \Big| \, 
x^T K y = 1, \#\{i \in [4] \,|\, v_i^T K y = 1\} = 0 \right\}$,
but $\#\{i \in [4] \,|\, v_i^T K y = 1\} = 0$ if and only if 
$v_1^T K y = v_2^T K y = v_3^T K y = 0$,
so $N(x) \cap S_0 = \{ y \in \mathbb{F}_2^{2\nu} \,|\, My = [1000]^T \}$,
where
\[ M = \begin{bmatrix} x^T K \\ v_1^T K \\ v_2^T K \\ v_3^T K \end{bmatrix}. \]
Since $x,v_1,v_2,v_3$ are linearly independent,
there exists a bijection from $N(x) \cap S_0$ to $\Ker T_M$.
Therefore we get $|N(x) \cap S_0| = 2^{2\nu-4}$.
Consequently,
	\begin{align*}
|N(x) \cap (S_0 \setminus (S \cup T))| 
&= |N(x) \cap S_0| - |N(x) \cap S_0 \cap (S \cup T)| \\
&= 2^{2\nu-4} - |N(x) \cap S_0 \cap S| - |N(x) \cap S_0 \cap T| \\
&= 2^{2\nu-4} - |N(x) \cap S| \\
&= 2^{2\nu-4} - 4.
	\end{align*}
On the other hand,
$|S_0 \setminus (S \cup T)| = 2^{2\nu-3} - 8$ by Lemma~4.5,
so we obtain $|N(x) \cap (S_0 \setminus (S \cup T))| = 
\frac{1}{2}| S_0 \setminus (S \cup T) |$.
We can determine the value of $|N(x) \cap S_2|$ similarly as above.
Observe 
\[ N(x) \cap S_2 = \left\{ y \in V(Sp(2\nu,2)) \, \Big| \, 
x^T K y = 1, \#\{ i \in [4] \,|\, v_i^T K y = 1 \} = 2 \right\}, \]
but $\#\{ i \in [4] \,|\, v_i^T K y = 1 \} = 2$ if and only if 
$\#\{ i \in [3] \,|\, v_i^T K y = 1 \} = 1$ or $2$.
Therefore
\[ N(x) \cap S_2 = \bigsqcup_{ \substack{ {\bm b} \in \mathbb{F}_2^{3}, \\
\wt({\bm b}) \in \{1,2\} }}
\left\{ y \in \mathbb{F}_2^{2\nu} \, \Bigg| \, 
\begin{bmatrix} x^T K \\ v_1^T K \\ v_2^T K \\ v_3^T K \end{bmatrix} y = 
\begin{bmatrix} 1 \\ {\bm b} \end{bmatrix} \right\} \]
and we get $\#\{ y \in \mathbb{F}_2^{2\nu} \,|\, My = [1{\bm b}^T]^T \} = 2^{2\nu-4}$
for a fixed ${\bm b}$.
Consequently,
$|N(x) \cap S_2| = 6 \cdot 2^{2\nu-4} = \frac{1}{2}|S_2|$,
so we can see that $S_4$ is a Godsil-McKay cell.

Finally,
we prove that $S_0 \setminus (S \cup T)$ is a Godsi-McKay cell.
Let $x \in S_0 \setminus (S \cup T)$ be a special vertex.
It is easy to see that $|N(x) \cap S| = |N(x) \cap T| = 0$.
Also, 
	\begin{align*}
& |N(x) \cap S_2| \\
=& \# \left\{ y \in V(Sp(2\nu,2)) \, \Big| \, 
x^T K y = 1, \#\{ i \in [4] \,|\, v_i^T K y = 1 \} = 2 \right\} \\
=& \sum_{ \substack{ {\bm b} \in \mathbb{F}_2^{3}, \\
\wt({\bm b}) \in \{1,2\} }}
\# \left\{ y \in \mathbb{F}_2^{2\nu} \, \Bigg| \, 
\begin{bmatrix} x^T K \\ v_1^T K \\ v_2^T K \\ v_3^T K \end{bmatrix} y = 
\begin{bmatrix} 1 \\ {\bm b} \end{bmatrix} \right\} \\
=& 6 \cdot 2^{2\nu-4} \\
=& \frac12 |S_2|.
	\end{align*}
Furthermore,
for $y \in S_4$,
	\begin{align*}
& |N(x) \cap S_4| \\
= & \frac{1}{|S_0 \setminus (S \cup T)|} |S_4| |N(y) \cap S_0 \setminus (S \cup T)|
 \tag{by Lemma~4.6} \\
= & \frac{1}{|S_0 \setminus (S \cup T)|} |S_4| \cdot \frac12|S_0 \setminus (S \cup T)| 
 \tag{by part (ii)} \\
= & \frac12 |S_4|.
	\end{align*}
Hence $S_0 \setminus (S \cup T)$ is a Godsil-McKay cell.
\end{proof}

Therefore on the orbit partition of $\Aut(X)_S$ on $V(X)$ ,
we obtain three switched symplectic graphs 
with Godsil-McKay cells $S,S_0 \setminus (S \cup T)$ and $S_4$.
Let $X^S$, $X^{S_0 \setminus (S \cup T)}$ and $X^{S_4}$
denote their switched graphs, respectively.
In general,
the set of edges deleted by Godsil-McKay switching with respect to a partition 
$\{C_1,\dots.C_t,D\}$ is
\[ \bigsqcup_{i=1}^{t} \bigsqcup_{x \in D}
\left\{ \{x,y\} \, \Bigg| \, y \in C_i, x \sim y, |N(x) \cap C_i| = \frac{1}{2}|C_i| \right\} \]
and the set of added edges is similarly
\[ \bigsqcup_{i=1}^{t} \bigsqcup_{x \in D}
\left\{ \{x,y\} \, \Bigg| \, y \in C_i, x \not\sim y, |N(x) \cap C_i| = \frac{1}{2}|C_i| \right\}. \]
Abiad and Haemers \cite{AH} proved that the partition $\{S,V(X) \setminus S\}$ 
is a Godsil-McKay partition with Godsil-McKay cell $D = V(X) \setminus S$, and 
constructed the switched symplectic graph that is not isomorphic to the original one.
The set of deleted edges to construct this switched symplectic graph 
by Abiad and Haemers is
\[ \bigsqcup_{x \in V(X) \setminus S}
\left\{ \{x,y\} \, \Bigg| \, y \in S, x \sim y, |N(x) \cap S| = \frac{1}{2}|S| \right\}, \]
but it is easy to see that $|N(x) \cap S| = \frac{1}{2}|S|$
if and only if $x \in S_2$. 
Therefore this is equal to
	\begin{equation}\label{4a}
\bigsqcup_{x \in S_2} \left\{ \{x,y\} \, \Big| \, y \in S, x \sim y \right\}.
	\end{equation}
On the other hand,
the set of deleted edges to construct $X^S$ is
\[ \bigsqcup_{P \in \{T,S_0 \setminus (S \cup T), S_2, S_4 \}} 
\bigsqcup_{x \in S} 
\left\{ \{x,y\} \, \Bigg| \, y \in P, x \sim y, |N(x) \cap P| = \frac{1}{2}|P| \right\}, \]
but we have already confirmed that for $x \in S$,
$|N(x) \cap P| = \frac{1}{2}|P|$ if and only if $P = S_2$ by Proposition~4.7-(i).
Therefore this is equal to (\ref{4a}) 
which is nothing but the one by Abiad and Haemers.
Similarly,
on the set of added edges,
	\begin{align*}
& \bigsqcup_{x \in V(X) \setminus S}
\left\{ \{x,y\} \, \Bigg| \, y \in S, x \not\sim y, |N(x) \cap S| = \frac{1}{2}|S| \right\} \\
=& \bigsqcup_{x \in S_2} \left\{ \{x,y\} \, \Big| \, y \in S, x \not\sim y \right\} \\
=& \bigsqcup_{P \in \{T,S_0 \setminus (S \cup T), S_2, S_4 \}} 
\bigsqcup_{x \in S} 
\left\{ \{x,y\} \, \Bigg| \, y \in P, x \not\sim y, |N(x) \cap P| = \frac{1}{2}|P| \right\},
	\end{align*}
so we can see the following:

\begin{cor} {\it
$X^S$ is isomorphic to the switched symplectic graph with respect to 
the Godsil-McKay partition $\{S, V(X) \setminus S\}$ with 
Godsil-McKay cell $V(X) \setminus S$. }
\end{cor}

We remark that for $x \in S_2(1,2)$ as a special vertex in $S_2$,
$N(x) \cap T = \{ v_2+v_3, v_3+v_1 \}$, that is,
$|N(x) \cap T| = \frac{2}{3}|T|$,
so $S_2$ is not a Godsil-McKay cell.
Therefore by Lemma~4.6,
we get $|N(x) \cap S_2| = \frac{2}{3}|S_2|$ for $x \in T$,
so $T$ is not a Godsil-McKay cell either. 

\section{Not being isomorphic} 

In this section,
we prove that the graphs in the five families
$X$, $X^{O(0,\nu,0)}$, $X^S$, $X^{S_0 \setminus (S \cup T)}$, $X^{S_4}$
are not isomorphic to each other.
To this end,
we consider the number of common neighbors of three vertices 
as an invariant for isomorphism.
First,
we investigate how the value of the number of common neighbors of three vertices 
changes after switching.
Next,
for each family,
by inspecting the non-zero minimum number of common neighbors of three vertices,
we prove that the graphs in different families are not isomorphic.

\subsection{Formulas that give the number of common neighbors of three vertices
in the switched graph}

Let $X$ be a graph and let $A,B$ be subsets of the vertex set $V(X)$
which are disjoint.
We define
\[ \cN_X[A|B] = \left\{ w \in V(X) \setminus (A \cup B) \,\Bigg|\,
\begin{split}
&w \sim a \, (\forall a \in A), \\
&w \not\sim b \, (\forall b \in B)
\end{split} \right\}. \]
Practically, we consider the case $|A \cup B| = 3$.
For example, for three distinct vertices $x,y,z$ in $V(X)$,
\[ \cN_X[\{x,y\}|\{z\}] = \{ w \in V(X) \setminus \{x,y,z\} \,|\,
w \sim x , \, w \sim y, \, w \not\sim z  \}, \]
but we will write $\cN_X[xy|z]$ instead of $\cN_X[\{x,y\}|\{z\}]$ for simplicity.

Let $\pi = \{C_1, \dots, C_t, C_{t+1} \}$ 
be the orbit partition of a group of automorphisms of $X$.
Assume that $\pi$ is a Godsil-McKay partition with Godsil-McKay cell $D = C_{t+1}$.
Then for any $i \in [t]$,
\[ \{ |N(x) \cap C_i| \,|\, x \in D \} 
= \{0\}, \left\{ \frac12 |C_i| \right\} \text{ or } \{|C_i|\}, \]
so we can decompose the index set $[t]$ depending on these values.
We define
	\begin{align*}
\mathcal{C}_0 &= \left\{ i \in [t] \, \Big| \, \{|N(x) \cap C_i| \,|\, x \in D \} = \{0\} \right\}, \\
\mathcal{C}_{\frac{1}{2}} &= \left\{ i \in [t] \, \Bigg| \, \{|N(x) \cap C_i| \,|\, x \in D \} 
= \left\{ \frac{1}{2}|C_i| \right\} \right\}, \\
\mathcal{C}_1 &= \left\{ i \in [t] \, \Big| \, \{|N(x) \cap C_i| \,|\, x \in D \} = \{|C_i|\} \right\}.
	\end{align*}
Then $[t] = \mathcal{C}_0 \sqcup \mathcal{C}_{\frac12} \sqcup \mathcal{C}_1$.
Let $X'$ be the switched graph 
with respect to $\pi$ with $D = C_{t+1}$.
To investigate the number of common neighbors of three vertices in $X'$,
we consider, for example, the case $x \in D=C_{t+1}$,
$y \in C_k$ and $z \in C_l$,
where $k \in \mathcal{C}_{\frac12}$ and $l \in \mathcal{C}_0 \cup \mathcal{C}_1$.
The set of pairs of vertices involved with switching is
\[ \bigsqcup_{i \in \mathcal{C}_{\frac12}} 
\left\{ \{v,w\} \, \Big| \, v \in D, w \in C_i \right\}, \]
so vertices in 
\[ \bigsqcup_{i \in \mathcal{C}_1} (C_i \cap \cN_X [xyz|\,]) \]
are also common neighbors of $x,y,z$ in $X'$.
On the other hand, in this case, vertices in 
\[ \bigsqcup_{i \in \mathcal{C}_{\frac12} \sqcup \{ t+1 \}} (C_i \cap \cN_X [xyz|\,]) \]
are no longer common neighbors of $x,y,z$ after switching.
However,
vertices in 
\[ \left( \bigsqcup_{i \in \mathcal{C}_{\frac12}} (C_i \cap \cN[yz|x]) \right)
\cup (D \cap \cN_X[xz|y]) \]
become new common neighbors after switching.
Consequently, we get
\[ |\mathcal{N}_{X'}[xyz|\,]|
= \sum_{i \in \mathcal{C}_1} \left|
 C_i \cap\mathcal{N}_{X}[xyz|\,] \right| 
+ \sum_{i \in \mathcal{C}_{\frac{1}{2}}} \left| C_i \cap\mathcal{N}_{X}[yz|x] \right|
+ | D \cap \mathcal{N}_{X}[xz|y] |. \]
For other cases,
we can investigate $\cN_{X'}[xyz|\,]$ by a similar argument as above,
so we get the following formulas on the number of 
common neighbors of three vertices in $X'$.

\begin{theo}{\it
Let $X$ be a graph and 
$\pi = \{C_1, \dots, C_t, C_{t+1} \}$ be the orbit partition of a group of automorphisms.
Assume that $\pi$ is a Godsil-McKay partition with a Godsil-McKay cell $D=C_{t+1}$.
Let $X'$ be the switched graph with respect to $\pi$.
Let $x,y,z$ be three distinct vertices in $V(X)$ and 
$i_x,i_y,i_z$ be indices to which $x,y,z$ belong, respectively.
Then for each of the following ten cases,
the values of $|\mathcal{N}_{X'}[xyz|\,]|$ are given in Table \ref{tb:1}.
\begin{longtable}[h]{lll}
\text{{\it (1)  $x,y,z \notin D$}}, & & \\
$\quad$ \text{{\it (i)} } $i_x,i_y,i_z \notin \mathcal{C}_{\frac12}$, & $\qquad$ &
\text{{\it (ii)} } $i_x \in \mathcal{C}_{\frac12}$ \text{ {\it and} } 
$i_y,i_z \notin \mathcal{C}_{\frac12}$, \\
$\quad$ \text{{\it (iii)} } $i_x, i_y \in \mathcal{C}_{\frac12}$ \text{ {\it and}} 
$i_z \notin \mathcal{C}_{\frac12}$, 
& $\qquad$ &
\text{{\it (iv)} } $i_x,i_y,i_z \in \mathcal{C}_{\frac12}$. \\
\text{{\it (2)  $x \in D$ and $y,z \notin D$}}, & & \\
$\quad$ \text{{\it (i)} } $i_y,i_z \notin \mathcal{C}_{\frac12}$, & $\qquad$ &
\text{{\it (ii)} } $i_y \in \mathcal{C}_{\frac12}$ \text{ {\it and} }
$i_z \notin \mathcal{C}_{\frac12}$, \\
$\quad$ \text{{\it (iii)} } $i_y,i_z \in \mathcal{C}_{\frac12}$. 
& $\qquad$ & \\
\text{{\it (3)  $x,y \in D$ and $z \notin D$}}, & & \\
$\quad$ \text{{\it (i)} } $i_z \notin \mathcal{C}_{\frac12}$, & $\qquad$ &
\text{{\it (ii)} } $i_z \in \mathcal{C}_{\frac12}$. \\
\text{{\it (4)  $x,y,z \in D$}}. & &
\end{longtable}
}
\end{theo}
{\small
\begin{center}
\def\arraystretch{2}
	$ \begin{array}{|c||c|} \hline
{\text {\it (1)-(i)}} & |\mathcal{N}_{X}[xyz|\,]| \\ \hline
{\text {\it (1)-(ii)}} &\displaystyle
\sum_{i \in \mathcal{C}_0 \sqcup \mathcal{C}_{\frac{1}{2}} \sqcup \mathcal{C}_1} \left| 
C_i \cap\mathcal{N}_{X}[xyz|\,] \right| 
+| D \cap \mathcal{N}_{X}[yz|x] | \\ \hline
{\text {\it (1)-(iii)}} &\displaystyle
\sum_{i \in \mathcal{C}_0 \sqcup \mathcal{C}_{\frac{1}{2}} \sqcup \mathcal{C}_1} \left| 
C_i \cap\mathcal{N}_{X}[xyz|\,] \right| 
+| D \cap \mathcal{N}_{X}[z|xy] | \\ \hline
{\text {\it (1)-(iv)}} &\displaystyle
\sum_{i \in \mathcal{C}_0 \sqcup \mathcal{C}_{\frac{1}{2}} \sqcup \mathcal{C}_1} \left| 
C_i \cap\mathcal{N}_{X}[xyz|\,] \right| 
+| D \cap \mathcal{N}_{X}[\,|xyz] | \\ \hline
{\text {\it (2)-(i)}} &\displaystyle
\sum_{i \in \mathcal{C}_0 \sqcup \mathcal{C}_1 \sqcup \{t+1\}} \left| 
C_i \cap\mathcal{N}_{X}[xyz|\,] \right| 
+ \sum_{i \in \mathcal{C}_{\frac{1}{2}}} \left| C_i \cap\mathcal{N}_{X}[yz|x] \right| \\ \hline
{\text {\it (2)-(ii)}} &\displaystyle
\sum_{i \in \mathcal{C}_0 \sqcup \mathcal{C}_1} \left| 
C_i \cap\mathcal{N}_{X}[xyz|\,] \right| 
+ \sum_{i \in \mathcal{C}_{\frac{1}{2}}} \left| C_i \cap\mathcal{N}_{X}[yz|x] \right|
+ | D \cap \mathcal{N}_{X}[xz|y] | \\ \hline
{\text {\it (2)-(iii)}} &\displaystyle
\sum_{i \in \mathcal{C}_0 \sqcup \mathcal{C}_1} \left| 
C_i \cap\mathcal{N}_{X}[xyz|\,] \right| 
+ \sum_{i \in \mathcal{C}_{\frac{1}{2}}} \left| C_i \cap\mathcal{N}_{X}[yz|x] \right|
+ | D \cap \mathcal{N}_{X}[x|yz] | \\ \hline
{\text {\it (3)-(i)}} &\displaystyle
\sum_{i \in \mathcal{C}_0 \sqcup \mathcal{C}_1 \sqcup \{t+1\}} \left| 
C_i \cap\mathcal{N}_{X}[xyz|\,] \right| 
+ \sum_{i \in \mathcal{C}_{\frac{1}{2}}} \left| C_i \cap\mathcal{N}_{X}[z|xy] 
\right| \\ \hline
{\text {\it (3)-(ii)}} &\displaystyle
\sum_{i \in \mathcal{C}_0 \sqcup \mathcal{C}_1} \left| 
C_i \cap\mathcal{N}_{X}[xyz|\,] \right| 
+ \sum_{i \in \mathcal{C}_{\frac{1}{2}}} \left| C_i \cap\mathcal{N}_{X}[z|xy] \right|
+ | D \cap \mathcal{N}_{X}[xy|z] | \\ \hline
{\text {\it (4)}} &\displaystyle
 \sum_{i \in \mathcal{C}_0 \sqcup \mathcal{C}_1 \sqcup \{t+1\}} \left| 
C_i \cap\mathcal{N}_{X}[xyz|\,] \right| 
+ \sum_{i \in \mathcal{C}_{\frac{1}{2}}} \left| C_i \cap\mathcal{N}_{X}[\,|xyz] 
\right| \\ \hline
	\end{array}\label{tb:1}$
\end{center}
\begin{center}
Table \ref{tb:1}. The number of common neighbors of three vertices in $X'$
\end{center}
}

\subsection{Investigating the non-zero minimum number of 
common neighbors of three vertices}

We use Table~\ref{tb:1} to 
investigate the number of common neighbors of three vertices for each family.
It is certainly difficult to determine all the possible values,
but our goal is to prove that the graphs in the five families
are not isomorphic to each other,
so it is sufficient to find an easier invariant for isomorphism.
From this point of view,
we calculate the non-zero minimum number of common neighbors of three vertices.

\begin{prop}{\it 
Let $X$ be the symplectic graph $Sp(2\nu,2)$ of order $2\nu$ and
let $x,y,z$ be three distinct vertices of $X$.
Then,
\[ |\mathcal{N}_X [xyz|\,]| = 
\begin{cases} 
0 &\text{if $x+y+z = \zero$,} \\
2^{2\nu-3} &\text{otherwise.}
\end{cases} \]
In particular,
the non-zero minimum number of common neighbors of three vertices in $Sp(2\nu,2)$ 
is $2^{2\nu-3}$. }
\end{prop}

\begin{proof}
First,
we assume $x+y+z = \zero$.
Suppose that there exists $w \in \mathcal{N}_X [xyz|\,]$.
Then $x^T K w = y^T K w = z^T K w = 1$,
but $1 = 1+1+1 = x^T K w + y^T K w + z^T K w = (x+y+z)^T K w = \zero^T K w = 0$.
This is a contradiction, so $|\mathcal{N}_X [xyz|\,]| = 0$.

Next,
we assume $x+y+z \neq \zero$.
Let
\[ M = \begin{bmatrix} x^T K \\ y^T K \\ z^T K \end{bmatrix}. \]
Then $\mathcal{N}_X [xyz|\,] = \{ w \in \mathbb{F}_2^{2\nu} \, | \, 
Mw = [111]^T \}$.
Since $x+y+z \neq \zero$,
$x,y,z$ are linearly independent,
so $\rank M = 3$.
Therefore the system of equations $Mw = {\bm 1}_3$ has a solution.
This implies that there is a bijection from $\cN_X [xyz|\,]$ to $\Ker T_M$.
The dimension of $\Ker T_M$ is $2\nu-3$,
so we get $|\mathcal{N}_X [xyz|\,]| = |\Ker T_M| = 2^{2\nu-3}$.
\end{proof}

\begin{prop}{\it 
Let $X$ be the symplectic graph $Sp(2\nu,2)$ of order $2\nu$ and let $X' = X^S$.
Take $x \in S_2(1,2)$, $y \in S_2(1,3)$ and set $z = x+y$.
Then $z \in S_2(2,3)$ and $|\mathcal{N}_{X'} [xyz|\,]| = 1$.
Therefore,
the non-zero minimum number of common neighbors of three vertices in $X^S$ 
is $1$.}
\end{prop}

\begin{proof}
Since
\[ z^T K v_i = x^T K v_i + y^T K v_i =
\begin{cases}
1 & \quad \text{if $i = 2,3$,} \\
0 & \quad \text{if $i = 1,4$,}
\end{cases} \]
we have $z \in S_2(2,3)$.
We recall that for $C \in \{ T, S_0 \setminus ( S \cup T ), S_2, S_4 \}$ and for $v \in S$,
\[ |N(v) \cap C| = \begin{cases}
0 & \quad \text{if $C = T$ or $S_0 \setminus (S \cup T)$,} \\
|C| & \quad \text{if $C = S_4$,} \\
\frac12 |C| & \quad \text{if $C = S_2$,}
\end{cases}
 \]
by Proposition~4.7-(i).
Thus,
	\begin{align*}
|\mathcal{N}_{X'} [xyz|\,]| &=  
|(T \cup ( S_0 \setminus (S \cup T)) \cup S_4 \cup S_2) \cap \mathcal{N}_{X} [xyz|\,]| \\ 
& \quad + |S \cap \mathcal{N}_{X} [\,|xyz]| \tag{by (1)-(iv) in Table~\ref{tb:1}} \\
&= |S \cap \mathcal{N}_{X} [\,|xyz]| \tag{by Proposition~5.2} \\
&= |\{ v_4 \}| \\
&=  1.
	\end{align*}
\end{proof}

Next,
we consider the non-zero minimum number of 
common neighbors of three vertices in $X^{O(0,\nu,0)}$.
If we decompose a vector $x \in \mathbb{F}_2^{2\nu}$
into $\nu$ blocks as follows:
\[ x = \begin{bmatrix}
x_1 \\ \vdots \\ x_{\nu}
\end{bmatrix} \qquad (x_i \in \mathbb{F}_2^2), \]
we can see
\[
\begin{bmatrix}
1100 & \cdots & 00 \\
0011 & \cdots & 00 \\
\vdots & \ddots & \vdots \\
0000 & \cdots & 11 
\end{bmatrix}
\begin{bmatrix}
x_1 \\ x_2 \\ \vdots \\ x_{\nu}
\end{bmatrix}
\equiv
\begin{bmatrix}
\wt(x_1) \\
\wt(x_2) \\
\vdots \\
\wt(x_{\nu})
\end{bmatrix} \quad (\MOD 2).
\]
Thus, for a vector $x \in \mathbb{F}_2^{2\nu}$,
there exists $j$ such that $x \in O(i,j,k)$ for some $i,k$ if and only if 
there exists a vector ${\bm b} \in \mathbb{F}_2^{\nu}$ whose weight is $j$ such that 
$(I_{\nu} \otimes [11]) x = {\bm b}$, 
so we can regard also an orbit of $\Aute$ 
as the solution set of a system of linear equations.

Recall that for an orbit $O(i,j,k)$ of $\Aute$ and for a vertex $v \in O(0,\nu,0)$,
\[ |N(v) \cap O(i,j,k)| = \begin{cases}
\frac{1}{2}|O(i,j,k)| &\text{ if $j \geq 1$,} \\
|O(i,j,k)| &\text{ if $j = 0$, $i : \odd$,} \\
0 &\text{ if $j = 0$, $i : \even$,}
\end{cases}
\]
by Proposition~3.5.
For three vertices $x,y,z$,
define the $(\nu+3) \times 2\nu$ matrix $M$ as follows:
\[ M = \begin{bmatrix}
x^T K \\
y^T K \\
z^T K \\
I_{\nu} \otimes [11]
 \end{bmatrix}. \]

\begin{lemm}{\it 
Let $X$ be the symplectic graph $Sp(2\nu,2)$ of order $2\nu$ and 
let $X' = X^{O(0,\nu,0)}$.
For three distinct vertices $x,y,z$,
$|\cN_{X'} [xyz|\,]|$ is a multiple of $2^{\nu-2}$.
}
\end{lemm}

\begin{proof}
For three distinct vertices $x,y,z$,
we consider two cases.

Case~1: Suppose $M$ has full rank.
We only consider the case (1)-(ii) of Theorem~5.1,
but on other cases,
we can consider similarly.
Assume that $x \in O(i,j,k)$ with $1 \leq j \leq \nu-1$,
$y \in O(l,0,m)$ and $z \in O(l',0,m')$.
According to Table~\ref{tb:1},
\[ |\mathcal{N}_{X'}[xyz|\,]|
= \sum_{i \in \mathcal{C}_0 \sqcup \mathcal{C}_1} \left| 
C_i \cap\mathcal{N}_{X}[xyz|\,] \right|
+ \sum_{i \in \mathcal{C}_{\frac12}} \left| 
C_i \cap\mathcal{N}_{X}[xyz|\,] \right|
+| D \cap \mathcal{N}_{X}[yz|x] |. \]
The first term is equal to 
$\# \{ w \in \mathbb{F}_2^{2\nu} \, | \, Mw = [111 \zero_{\nu}^T]^T \}$.
Since $M$ has full rank, it is $2^{2\nu-(\nu+3)} = 2^{\nu-3}$.
The second term is equal to
\[
\sum_{ \substack{ {\bm b} \in \mathbb{F}_2^{\nu}, \\ 
1 \leq \wt({\bm b}) \leq \nu -1 }}
\# \left\{
w \in \mathbb{F}_2^{2\nu} \, \Bigg| \, Mw = 
\begin{bmatrix} 1 \\ 1 \\ 1 \\ {\bm b} \end{bmatrix} \right\},
\]
but $M$ has full rank, so it is $(2^{\nu} - 2) \cdot 2^{\nu-3}$.
The third term is equal to 
$\# \{ w \in \mathbb{F}_2^{2\nu} \setminus \{ x \} \, | \, Mw = [011 {\bm 1}_{\nu}^T]^T \}$,
but since $x \in O(i,j,k)$ with $1 \leq j \leq \nu-1$,
$(I_{\nu} \otimes [11]^T)x \neq {\bm 1}_{\nu}$.
Thus, $x$ is not a solution of $Mw = [011 {\bm 1}_{\nu}^T]^T$,
so we get $| D \cap \mathcal{N}_{X}[yz|x] | = 2^{\nu-3}$.
Consequently,
\[ |\mathcal{N}_{X'}[xyz|\,]| = 
2^{\nu-3} + (2^{\nu} - 2) \cdot 2^{\nu-3} + 2^{\nu-3} = 2^{2\nu-3}. \]
In particular,
$|\cN_{X'} [xyz|\,]|$ is a multiple of $2^{\nu-2}$.
(Note that on other cases,
if $M$ has full rank,
then $|\cN_{X'} [xyz|\,]| = 2^{2\nu-3}$.)

Case~2: Suppose $M$ does not have full rank,
that is,
$\rank M = \nu, \nu +1$ or $\nu+2$.
We argue similarly to the case~1.
For each case,
there exist proper subsets $A,B,A',B'$ of $V(X)$ such that
\[ |\mathcal{N}_{X'}[xyz|\,]|
= \sum_{i \in \mathcal{C}_0 \sqcup \mathcal{C}_1} \left|
C_i \cap\mathcal{N}_{X}[xyz|\,] \right|
+ \sum_{i \in \mathcal{C}_{\frac12}} \left| 
C_i \cap\mathcal{N}_{X}[A|B] \right|
+| D \cap \mathcal{N}_{X}[A'|B'] | \]
and we can confirm that 
$x,y,z,\zero$ are not a solution of the system of linear equations 
determined by each term.
Thus,
each term is a multiple of $2^{2\nu-\rank M}$,
but $\rank M \leq \nu + 2$ in this case.
Therefore, $|\cN_{X'} [xyz|\,]|$ is a multiple of $2^{\nu-2}$.
\end{proof}

Fortunately,
we can take three vertices that give $|\cN_{X'} [xyz|\,]| = 2^{\nu-2}$.

\begin{prop}{\it
Let $X$ be the symplectic graph $Sp(2\nu,2)$ of order $2\nu$ and let $X' = X^{O(0,\nu,0)}$.
Then the non-zero minimum number of common neighbors 
of three vertices in $X^{O(0,\nu.0)}$ is $2^{\nu-2}$. }
\end{prop}

\begin{proof}
Pick $x = [101000 \zero_{2\nu-6}]^T$,
$y = [100010 \zero_{2\nu-6}] \in O(0,2,\nu-2)$ and set $z = x+y$.
Then $z = [001010 \zero_{2\nu-6}] \in O(0,2,\nu-2)$.
Thus, by the case (1)-(iv) of Theorem~5.1, we get
$|\cN_{X'} [xyz|\,]| = |D \cap \cN_{X} [\,|xyz]|$.
Also,
	\begin{align*}
D \cap \cN_{X} [\,|xyz] &= 
\{ w \in \mathbb{F}_2^{2\nu} \setminus \{ x,y,z,\zero \} \, | \, 
Mw = [000 {\bm 1}_{\nu}^T]^T \} \\
&=  \{ w \in \mathbb{F}_2^{2\nu} \, | \, Mw = [000 {\bm 1}_{\nu}^T]^T \} \\
&= \left\{ w \in \mathbb{F}_2^{2\nu} \, \Bigg| \,
\begin{bmatrix}
x^T K \\
y^T K \\
I_{\nu} \otimes [11]
\end{bmatrix} w = 
\begin{bmatrix} 0 \\ 0 \\ {\bm 1}_{\nu} \end{bmatrix} \right\},
	\end{align*}
and the matrix $\begin{bmatrix}
x^T K \\
y^T K \\
I_{\nu} \otimes [11]
\end{bmatrix}$ has rank $\nu + 2$, which has full rank.
Thus, we see that $|\cN_{X'} [xyz|\,]| = |D \cap \cN_{X} [\,|xyz]| = 2^{\nu-2}$.
\end{proof}

Next,
we consider the family $X^{S_4}$.
Recall that for an orbit $C \in \{ S,T,S_0 \setminus \{ S \cup T \}, S_2 \}$ of $\Aut(X)_S$
and for a vertex $v \in S_4$,
\[ |N(v) \cap C| = \begin{cases}
0 &\text{if $C = T$,} \\
|C| &\text{if $C = S$,} \\
\frac12 |C| &\text{if $C = S_0 \setminus \{S \cup T\}$ or $S_2$,}
\end{cases}
\]
by Proposition~4.7-(ii).
Also, for three vertices $x,y,z \in V(X)$,
we redefine the matrix $M$ as follows:
\[ M = \begin{bmatrix}
x^T K \\
y^T K \\
z^T K \\
v_1^T K \\
v_2^T K \\
v_3^T K
\end{bmatrix}. \]

\begin{lemm}{\it
Let $X$ be the symplectic graph $Sp(2\nu,2)$ of order $2\nu$ and 
let $X' = X^{S_4}$.
Assume that $x,y,z$ are three distinct vertices in $X'$.
If $\cN_{X'} [xyz|\,]$ is nonempty,
then $|\cN_{X'} [xyz|\,]| \geq 2^{2\nu-5}$. }
\end{lemm}

\begin{proof}
For three distinct vertices $x,y,z$,
we consider two cases.

Case~1: Suppose $M$ has full rank.
We consider the case (2)-(i) of Theorem~5.1 for example,
but we can consider similarly on other cases too.
Assume that $x \in D = S_4$ and $y,z \in S \cup T$.
By the case (2)-(i) of Theorem~5.1,
	\begin{align*}
|\cN_{X'}[xyz|\,]| & \geq |S_2 \cap \cN_{X}[yz|x]| \\
& = \sum_{ \substack{ {\bm b} \in \mathbb{F}_2^3, \\ \wt({\bm b}) \in \{1,2\} }}
\# \left\{
w \in \mathbb{F}_2^{2\nu} \setminus \{x,y,z,\zero\} \, \Bigg| \,
Mw = \begin{bmatrix} 0 \\ 1 \\ 1 \\ {\bm b} \end{bmatrix}
\right\}.
	\end{align*}
Clearly,
$y,z,\zero$ are not a solution of $Mw = [011 {\bm b}^T]^T$.
Also,
since $x \in S_4$,
\[ \begin{bmatrix} v_1^T K \\
v_2^T K \\
v_3^T K
\end{bmatrix} x =
\begin{bmatrix}
1 \\
1 \\
1
\end{bmatrix}. \]
Thus, $x$ is not a solution of $Mw = [011 {\bm b}^T]^T$, either.
Therefore,
since $M$ has full rank, we see that
\[ |\cN_{X'}[xyz|\,]| \geq 6 \cdot 2^{2\nu-6} \geq 2^{2\nu-5}. \]
Accordingly, $\cN_{X'} [xyz|\,]$ is always nonempty in this case.

Case~2: Suppose $M$ does not have full rank,
that is, $\rank M = 3,4$ or $5$.
We consider the case (1)-(ii) of Theorem~5.1 for example and 
we can argue similarly on other cases except three cases (2)-(iii), (3)-(ii) and (4). 
Assume that $x \in (S_0 \setminus (S \cup T)) \cup S_2$ and $y,z \in S \cup T$.
By the case (1)-(ii) of Theorem~5.1,
	\begin{align*}
\cN_{X'}[xyz|\,] = ((S \cup T) \cap \cN_{X}[xyz|\,]) \sqcup 
((S_0 \setminus (S \cup T)) \cap \cN_{X}[xyz|\,]) \\ \sqcup 
( S_2 \cap \cN_{X}[xyz|\,]) \sqcup 
( S_4 \cap \cN_{X}[yz|x]).
	\end{align*}
Since $y,z \in S_0$, $(S \cup T) \cap \cN_{X}[xyz|\,] = \emptyset$.
By $\cN_{X'}[xyz|\,] \neq \emptyset$,
one of $(S_0 \setminus (S \cup T)) \cap \cN_{X}[xyz|\,]$,
$S_2 \cap \cN_{X}[xyz|\,]$ or
$S_4 \cap \cN_{X}[yz|x]$ is nonempty.
We suppose $(S_0 \setminus (S \cup T)) \cap \cN_{X}[xyz|\,] \neq \emptyset$ first.
We see that
\[ (S_0 \setminus (S \cup T)) \cap \cN_{X}[xyz|\,]
= \{ w \in \mathbb{F}_2^{2\nu} \setminus (\{ x,y,z,\zero \} \cup S \cup T) \, | \,
Mw = [111000]^T \}, \]
but any vector in $\{ x,y,z,\zero \} \cup S \cup T$
is not a solution of $Mw = [111000]^T$ since $y,z \in S_0$.
On the other hand,
$Mw = [111000]^T$ has a solution,
so
	\begin{align*}
|(S_0 \setminus (S \cup T)) \cap \cN_{X}[xyz|\,]|
&= \# \{ w \in \mathbb{F}_2^{2\nu} \, | \, Mw = [111000]^T \} \\
&= 2^{2\nu - \rank M} \\
&\geq 2^{2\nu-5}.
	\end{align*}
Next,
we suppose $S_2 \cap \cN_{X}[xyz|\,] \neq \emptyset$.
Then there exists a vector ${\bm b}$ whose weight is 1 or 2 such that
\[ \{ w \in \mathbb{F}_2^{2\nu} \setminus \{ x,y,z,\zero \} \, | \, 
Mw = [111 {\bm b}^T]^T \} \neq \emptyset . \]
Since $x,y,z,\zero$ are not a solution of $Mw = [111 {\bm b}^T]^T$,
	\begin{align*}
| S_2 \cap \cN_{X}[xyz|\,]|
&\geq \# \{ w \in \mathbb{F}_2^{2\nu} \, | \, Mw = [111 {\bm b}^T]^T \} \\
&= 2^{2\nu - \rank M} \\
&\geq 2^{2\nu-5}.
	\end{align*}
Finally,
we suppose $S_4 \cap \cN_{X}[yz|x] \neq \emptyset$.
Then 
\[ \{ w \in \mathbb{F}_2^{2\nu} \setminus \{ x,y,z,\zero \} \, | \, Mw = [011111]^T \}
 \neq \emptyset, \]
but $x,y,z,\zero$ are not a solution of $Mw = [011111]^T$ since $x \notin S_4$.
Thus,
\[ |S_4 \cap \cN_{X}[yz|x]| = 2^{2\nu-\rank M} \geq 2^{2\nu-5}. \]
In this way,
we can basically prove that
for appropriate subsets $A$, $A'$, $A''$,
$B$, $B'$, $B'' \subset V(X)$ determined by each case of Table~\ref{tb:1},
	\begin{itemize}
	\item $(S \cup T) \cap \cN_{X}[xyz|\,] = \emptyset$,
	\item If $(S_0 \setminus (S \cup T)) \cap \cN_{X}[A|B] \neq \emptyset$,
then $|(S_0 \setminus (S \cup T)) \cap \cN_{X}[A|B]| \geq 2^{2\nu-5}$,
	\item If $S_2 \cap \cN_{X}[A'|B'] \neq \emptyset$,
then $|S_2 \cap \cN_{X}[A'|B']| \geq 2^{2\nu-5}$,
	\item If $S_4 \cap \cN_{X}[A''|B''] \neq \emptyset$,
then $|S_4 \cap \cN_{X}[A''|B'']| \geq 2^{2\nu-5}$,
	\end{itemize}
so we can see that $|\cN_{X'} [xyz|\,]| \geq 2^{2\nu-5}$ as a result.
However,
if $x,y,z \in S_2 \cup S_4$,
then $(S \cup T) \cap \cN_{X}[xyz|\,] \neq \emptyset$ can occur.
Thus,
we need other arguments on the cases (2)-(iii), (3)-(ii) and (4).
	\begin{enumerate}[(I)]
	\item The case (2)-(iii), especially, $x \in D=S_4$ and $y,z \in S_2$.
		\begin{itemize}
		\item If $y,z \in S_2(i,j)$,
then $(S \cup T) \cap \cN_{X}[xyz|\,] = \{v_i, v_j\}$,
but $S_0 \cap \cN_{X}[yz|x] \cap \SpanS = \{ v_i+v_k, v_j+v_k \}$
for $k \in [4] \setminus \{i,j\}$,
so 
\[ |\cN_{X'} [xyz|\,]| \geq 2 + (2^{2\nu - \rank M} - 2) \geq 2^{2\nu-5}. \]
		\item If $y \in S_2(i,j)$ and $z \in S_2(i,k)$ for distinct indices $i,j,k$,
then $(S \cup T) \cap \cN_{X}[xyz|\,] = \{v_i\}$,
but $S_0 \cap \cN_{X}[yz|x] \cap \SpanS = \{ v_j+v_k \}$,
so 
\[ |\cN_{X'} [xyz|\,]| \geq 1 + (2^{2\nu - \rank M} - 1) \geq 2^{2\nu-5}. \]
		\item If $y \in S_2(i,j)$ and $z \in S_2(k,l)$ for distinct indices $i,j,k,l$,
then $(S \cup T) \cap \cN_{X}[xyz|\,] = \emptyset$.
Thus,
this case is no problem because we can use a ``basis'' argument.
		\end{itemize}
	\item The case (3)-(ii), especially, $x,y \in D=S_4$ and $z \in S_2$.
Assume that $z \in S_2(i,j)$.
Then $(S \cup T) \cap \cN_{X}[xyz|\,] = \{v_i,v_j\}$,
but $S_0 \cap \cN_{X}[z|xy] \cap \SpanS = \{ v_i+v_k, v_j+v_k \}$
for $k \in [4] \setminus \{i,j\}$,
so 
\[ |\cN_{X'} [xyz|\,]| \geq 2 + (2^{2\nu - \rank M} - 2) \geq 2^{2\nu-5}. \]
	\item The case (4).
Then $(S \cup T) \cap \cN_{X}[xyz|\,] = S$,
but $S_0 \cap \cN_{X}[\,|xyz] \cap \SpanS = T \cup \{ \zero \}$,
so 
\[ |\cN_{X'} [xyz|\,]| \geq 4 + (2^{2\nu - \rank M} - 4) \geq 2^{2\nu-5}. \]
	\end{enumerate}
Consequently,
we can get the desired inequality for all cases.
\end{proof}

\begin{prop}{\it
Let $X$ be the symplectic graph $Sp(2\nu,2)$ of order $2\nu$ and let $X' = X^{S_4}$.
Then the non-zero minimum number of common neighbors 
of three vertices in $X^{S_4}$ is $2^{2\nu-5}$. }
\end{prop}

\begin{proof}
We take $x \in S_2(1,2)$ and $y \in S_2(2,3)$ and set $z=x+y$.
Then $z \in S_2(1,3)$,
so by the case (1)-(iv) of Theorem~5.1,
we get
\[ |\cN_{X'}[xyz|\,]| = | (S_0 \cup S_2) \cap \cN_X[xyz|\,] | + |S_4 \cap \cN_{X}[\,|xyz]|. \]
Since $x+y+z = \zero$,
the first term of the right hand side is zero by Proposition~5.2.
Thus,
	\begin{align*}
|\cN_{X'}[xyz|\,]| &= |S_4 \cap \cN_{X}[\,|xyz]| \\
&= \# \{ w \in \mathbb{F}_2^{2\nu} \, | \, Mw = [000111]^T \} \\
&= \# \left\{ 
w \in \mathbb{F}_2^{2\nu} \, \Bigg| \, \begin{bmatrix} 
x^T K \\ y^T K \\ v_1^T K \\ v_2^T K \\ v_3^T K \end{bmatrix} w =
\begin{bmatrix} 0 \\ 0 \\ 1 \\ 1 \\ 1 \end{bmatrix} \right\},
	\end{align*}
but $x,y,v_1,v_2,v_3$ are linearly independent.
Therefore,
$\begin{bmatrix} x^T K \\ y^T K \\ v_1^T K \\ v_2^T K \\ v_3^T K \end{bmatrix}$ 
has full rank,
so we see that $|\cN_{X'}[xyz|\,]| = 2^{2\nu-5}$.
\end{proof}

Finally,
we consider the family $X^{S_0 \setminus (S \cup T)}$.
Recall that for an orbit $C \in \{ S,T, S_2, S_4 \}$ of $\Aut(X)_S$
and for a vertex $v \in S_4$,
\[ |N(v) \cap C| = \begin{cases}
0 &\text{if $C = S$ or $T$,} \\
\frac12 |C| &\text{if $C = S_2$ or $S_4$,}
\end{cases}
\]
by Proposition~4.7-(iii).
We prove that the non-zero minimum number of common neighbors 
of three vertices in $X^{S_0 \setminus (S \cup T )}$ is $2^{2\nu-5}-2$,
but its proof is similar to the one in $X^{S_4}$ basically,
that is,
we can see the following for appropriate subsets 
$A$,$A'$, $A''$, $B$, $B'$, $B'' \subset V(X)$
determined by each case of Table~\ref{tb:1}.
	\begin{enumerate}[(I)]
	\item When $M$ has full rank, we see that
\[ |\cN_{X'}[xyz|\,]| \geq |S_2 \cap \cN_X [A|B]| \geq 2^{2\nu-5} \geq 2^{2\nu-5}-2. \]
	\item When $M$ does not have full rank,
we can prove the following 
except the case (1)-(iv).
		\begin{itemize}
		\item $(S \cup T) \cap \cN_{X}[xyz|\,] = \emptyset$,
		\item If $S_2 \cap \cN_{X}[A|B] \neq \emptyset$,
then $|S_2 \cap \cN_{X}[A|B]| \geq 2^{2\nu-5}$,
		\item If $S_4 \cap \cN_{X}[A'|B'] \neq \emptyset$,
then $|S_4 \cap \cN_{X}[A'|B']| \geq 2^{2\nu-5}$,
		\item If $(S_0 \setminus (S \cup T)) \cap \cN_{X}[A''|B''] \neq \emptyset$,
then $|(S_0 \setminus (S \cup T)) \cap \cN_{X}[A''|B'']| \geq 2^{2\nu-5}$.
		\end{itemize}
	\end{enumerate}
The exceptional case (1)-(iv) is proved as follows.

\begin{lemm}
Let $X$ be the symplectic graph $Sp(2\nu,2)$ of order $2\nu$ and 
let $X' = X^{S_0 \setminus (S \cup T )}$.
Suppose that $x,y,z \in S_2 \cup S_4$ and $M$ does not have full rank.
Then $|\cN_{X'}[xyz|\,]| \geq 2^{2\nu-5} - 2$.
\end{lemm}

\begin{proof}
By the case (1)-(iv) of Theorem~5.1,
\[ |\cN_{X'}[xyz|\,]| \geq 
|(S \cup T) \cap \cN_{X}[xyz|\,]| + |(S_0 \setminus (S \cup T) \cap \cN_{X}[\,|xyz]|. \]
Since $x,y,z$ are not a solution of the system of equations $Mw = \zero_6$,
\[ (S_0 \setminus (S \cup T)) \cap \cN_{X}[\,|xyz]
= \{ w \in F_2^{2\nu} \,|\, Mw = \zero_6 \} \setminus \SpanS. \]
Thus,
	\begin{align*}
|(S_0 \setminus (S \cup T)) \cap \cN_{X}[\,|xyz]| 
&= |\Ker T_M| - |\Ker T_M \cap \SpanS| \\
&\geq 2^{2\nu-5} - |\Ker T_M \cap \SpanS|,
	\end{align*}
but since $x,y,z \in S_2 \cup S_4$,
$\dim \Ker T_M \cap \SpanS \leq 2$.
If $\dim \Ker T_M \cap \SpanS = 0$ or $1$,
then we can get the desired inequality,
so we assume $\dim \Ker T_M \cap \SpanS = 2$.
We aim to prove that $|(S \cup T) \cap \cN_{X}[xyz|\,]| \geq 2$ in this case.
(Actually, we can prove $|(S \cup T) \cap \cN_{X}[xyz|\,]| = 4$.)
Observe that there exist two distinct indices $i,j \in [4]$ such that 
$Mv_i \neq \zero_6$ and $Mv_j \neq \zero_6$.
Since $|\Ker T_M \cap \SpanS| = 4$,
there exist two distinct indices $k,l \in [4]$ such that 
$v_k + v_l \in \Ker T_M \cap \SpanS$.
We can assume $k=1, l=2$ without loss of generality.
It is sufficient to check the following two cases.

Case 1:
Suppose that $\Ker T_M \cap \SpanS = \langle v_1+v_2, v_1 \rangle$.
Observe $x,y,z \in S_2(3,4)$,
and we see that
\[ |(S \cup T) \cap \cN_{X}[xyz|\,]| = \# \{ v_3,v_4,v_1+v_3,v_2+v_3 \} = 4. \]

Case 2:
Suppose that $\Ker T_M \cap \SpanS = \langle v_1+v_2, v_2+v_3 \rangle$.
If we suppose $z \in S_2(i,j)$,
then $z^T K (v_i+v_k) = 1$ for $k \in [4] \setminus \{i,j\}$,
but this is a contradiction.
Thus,
$x,y,z$ have to be in $S_4$.
Consequently,
\[ |(S \cup T) \cap \cN_{X}[xyz|\,]| = \# \{ v_1,v_2,v_3,v_4 \} = 4. \]
\end{proof}

\begin{prop}{\it
Let $X$ be the symplectic graph $Sp(2\nu,2)$ of order $2\nu$ and 
let $X' = X^{S_0 \setminus (S \cup T)}$.
Then the non-zero minimum number of common neighbors 
of three vertices in $X^{S_0 \setminus (S \cup T)}$ is $2^{2\nu-5}-2$. }
\end{prop}

\begin{proof}
Take $y \in S_2(1,2)$ and $z \in S_2(3,4)$ and set $x = y+z$.
Then $x \in S_4$ and by using the case (1)-(iv) of Theorem~5.1,
we can check $|\cN_{X'}[xyz|\,]| = 2^{2\nu-5}-2$.
\end{proof}

Summarizing this subsection,
we get the following:

\begin{theo}{\it 
Let $X$ be the symplectic graph $Sp(2\nu,2)$ of order $2\nu$.
The non-zero minimum numbers of common neighbors of three distinct vertices
for each graphs are given in the Table~5.2.
	\begin{center}
	\begin{tabular}{|c|c|c|c|c|} \hline
$X$ & $X^S$ & $X^{O(0,\nu,0)}$ & $X^{S_4}$ & $X^{S_0 \setminus (S \cup T)}$ \\ \hline
$2^{2\nu-3}$ & $1$ & $2^{\nu-2}$ & $2^{2\nu-5}$ & $2^{2\nu-5}-2$ \\ \hline
	\end{tabular}
	\end{center}
	\begin{center}\emph{
Table 5.2. The non-zero minimum numbers of common neighbors of 
three distinct vertices}
	\end{center}

In particular,
the five graphs
$X,X^{O(0,\nu,0)}, X^S, X^{S_4}, X^{S_0 \setminus (S \cup T)}$
are not isomorphic to each other for a fixed $\nu$.
}
\end{theo}

\end{document}